\documentclass[11pt,reqno,twoside]{amsart}
\usepackage{graphicx}

\usepackage{subfigure}
\usepackage{amsmath,mathdots,amssymb,latexsym,amsthm,mathrsfs,epsfig}
\usepackage{multirow}
\usepackage{breqn}
\usepackage{float}
\usepackage{cite}
\usepackage{enumitem}
\usepackage{environ}
\usepackage{hyperref}
\usepackage{multicol}
\usepackage{mathtools}
\NewEnviron{myequation}{%
	\begin{equation}
		\scalebox{1.1}{$\BODY$}
	\end{equation}
}
\usepackage{color}
\usepackage{soul}
\renewcommand{\baselinestretch}{1.2}
\makeatletter
\newcommand{\single}{\let\CS=\@currsize\renewcommand{\baselinestretch}{1.1}\tiny\CS}
\newcommand{\singb}{\let\CS=\@currsize\renewcommand{\baselinestretch}{1}\tiny\CS}
\newcommand{\singa}{\let\CS=\@currsize\renewcommand{\baselinestretch}{1.2}\tiny\CS}
\newcommand{\oneandahalfspacing}{\let\CS=\@currsize\renewcommand{\baselinestretch}{1.5}\tiny\CS}
\newcommand{\singlespacing}{\let\CS=\@currsize\renewcommand{\baselinestretch}{1.6}\large\CS}
\newcommand{\bc}{\begin{center}}
	\newcommand{\ec}{\end{center}}
\newcommand{\be}{\begin{eqnarray}}
	\newcommand{\ee}{\end{eqnarray}}

\newcommand{\Hom}{\operatorname{Hom}}

\newcommand{\Dim}{\operatorname{dim}}
\newcommand{\Irr}{\operatorname{Irr}}
\newcommand{\Ind}{\operatorname{Ind}}

\newcommand{\Span}{\operatorname{Span}}
\newcommand{\Det}{\operatorname{det}}

\newcommand{\beano}{\begin{eqnarray*}}
	\newcommand{\eeano}{\end{eqnarray*}}

\newcommand{\ba}{\begin{array}}
	\newcommand{\ea}{\end{array}}

\makeatletter
\newtheorem{theorem}{Theorem}[section]
\newtheorem{corollary}[theorem]{Corollary}
\newtheorem{lemma}[theorem]{Lemma}
\newtheorem{proposition}[theorem]{Proposition}

\theoremstyle{definition}
\newtheorem{definition}{Definition}[section]

\numberwithin{equation}{section}

\DeclareMathOperator{\GL}{GL}

\DeclareMathOperator{\Par}{P}
\DeclareMathOperator{\G}{G}
\DeclareMathOperator{\F}{F}
\DeclareMathOperator{\N}{N}
\DeclareMathOperator{\Ha}{H}
\DeclareMathOperator{\M}{M}
\DeclareMathOperator{\V}{V}
\DeclareMathOperator{\JH}{JH}

\allowdisplaybreaks
\begin{document}
	
	\title[A NOTE ON JACQUET MODULES OF GENERAL LINEAR GROUPS]{A NOTE ON JACQUET MODULES OF GENERAL LINEAR GROUPS}
	
	%%%%%%    Information for first author
	
	\author[PREM DAGAR]{PREM DAGAR}
	\address{Department of Mathematics, Indian Institute of Technology Roorkee, Uttarakhand, 247667, Bharat (India)}
	\email{prem\textunderscore d@ma.iitr.ac.in}
	%\thanks{$^{*}$Research supported by the Ministry of Human Resource Development, India.}

	%%%%%%   Information for second author
	
	\author[MAHENDRA KUMAR VERMA]{MAHENDRA KUMAR VERMA}
	\address{Department of Mathematics, Indian Institute of Technology Roorkee, Uttarakhand, 247667,  Bharat (India)}
	\email{mahendraverma@ma.iitr.ac.in}

\begin{abstract} 
	Let $\F$ be a non-Archimedean local field. Consider $\G_n:= \GL_n(\F)$ and let $\M:= \G_l\times \G_{n-l}$ be a maximal Levi subgroup of $\G_n$.~In this article, we compute the semisimplified Jacquet module of representations of $G_n$ with respect to the maximal Levi subgroup $\M$, belonging to a particular category of representations. Utilizing our results, we prove that the Jacquet module is multiplicity-free for a specific subcategory of representations. Our findings are based on the Zelevinsky classification.

\end{abstract}
	\maketitle

\section{Introduction}\label{sec1}
Let $\G_n=\GL_n(\F)$ be a general linear group over a non-Archimedean local field $\F$. For a group $\G$, let $\mathcal{M}(\G_n)$ denote the category of all smooth and admissible complex representations of the finite length of $G$. A valuable tool for studying the representations of $\G_n$ is the Jacquet module. Let $\Par$ denote a parabolic subgroup of $\G_n$ with a Levi decomposition $\Par = \M\N$, where $\M$ is a Levi subgroup and $\N$ is the corresponding unipotent subgroup. Let $(\pi, \V)$ be a smooth and admissible complex representation of the group $\G_n$.~~The Jacquet module of representation $\V$ with respect to $\N$ is $$ \V_{\N}=\V/\Span \{\pi(n)v-v; v\in \V,n\in \N\}. $$ It is the largest quotient of $\V$ on which $\N$ acts trivially. Here, the Levi subgroup $\M$ acts naturally  on $\V_{\N}$. So, the  Jacquet module functor, or simply say, Jacquet functor, is a functor from the category $\mathcal{M}(\G_n)$ of $\G_n$ to the category $\mathcal{M}(\M)$ of $\M$. Let denote this functor by  $r_{\M}^{\G_n}$. The Jacquet functor is an exact functor that is left adjoint to the induction functor, i.e., $$\Hom_{\G_n}(\pi,\Ind_{\Par}^{\G_n}(\sigma))\cong \Hom_{\M}(r_{\M}^{\G_n}(\pi),\sigma),$$ for a smooth representation $\sigma$ of the Levi subgroup $\M$. Here, $\Ind_{\Par}^{\G_n} (\sigma)$ denote the representation of $\G_n$, parabolically induced by $\sigma$ from $\Par$. Note that if the semisimplification of $r_{\M}^{\G_n}(\pi)$ is multiplicity free, i.e., any component appears as a quotient or as a subrepresentation at most one time in $r_{\M}^{\G_n}(\pi)_{ss}$, then $r_{\M}^{\G_n}(\pi)$ is also multiplicity free. So, we identify $r_{\M}^{\G_n}(\pi)$ with its semisimplification. We consider normalized  Jacquet and induction functors, i.e., they carry unitarizable representations to unitarizable ones. The Frobenius reciprocity indicates the importance of the knowledge of the Jacquet module of a representation. Much less is known about the Jacquet module of an arbitrary representation of $\G_n$. Within the framework of the Zelevinsky classification, we analyze the Jacquet module of some specific class of representations of the general linear group with respect to any maximal Levi subgroup. An important question arises here: when does the Jacquet module of a representation exhibit multiplicity-free characteristics. We may not fully resolve the question, but our analysis offers a partially qualified answer by giving a subcategory whose representations have a multiplicity-free Jacquet modules.

\begin{definition}
	Let $\Ha$ and $\Ha'$ be two reductive groups over field $\F$.~We call a functor $\mathcal{F}:\mathcal{M}(\Ha)\rightarrow \mathcal{M}(\Ha')$ a multiplicity free functor if for any	irreducible admissible representation $\pi\in \mathcal{M}(\Ha)$, the representation $\mathcal{F}(\pi)$ is multiplicity free representation, i.e., $\dim \Hom_{\Ha'}(\mathcal{F}(\pi), \tau )\leq 1,$  for any irreducible representation $\tau\in \mathcal{M}(\Ha').$
	
\end{definition}

In this article, we prove that the Jacquet functor $r_{\M}^{\G_n}$ is multiplicity-free for a subcategory of irreducible representations of $\G_n$. This problem's inspiration originates from the work of Aizenbud and  Gourevitch \cite{aizenbud2012multiplicity}. They proved that the Jacquet functor $r_{\M}^{\G_n}:\mathcal{M}(\G_{n+k})\rightarrow \mathcal{M}(\G_n\times\G_k)$ is multiplicity free for $k=1 \hspace{0.1cm}\text{or} \hspace{0.1cm}2$.
Some more multiplicity-free problems can be found in \cite{adler2006certain,aizenbud2009multiplicity,aizenbud2010multiplicity,aizenbud2009generalized,shalika1974multiplicity,sun2012multiplicity}.
To familiarize themselves with key notations and the representations central to our discourse, readers are directed to Section \ref{3.1}. Zelevinsky \cite{z80} computed the Jacquet module of the representations of type $Z( \Delta) $. Later, Kret and  Lapid \cite{kret2012jacquet} computed the Jacquet module of ladder representations. In this article, we compute the  Jacquet module of any representations of type $\pi=Z( \Delta_1)\times\cdots\times Z( \Delta_r)\in\mathcal{M}(\G_n)$ and prove that if the functor $ r_{\M}^{\G_n}$ is restricted to $\mathcal{M}_{\Irr}(\G_n)$ (see subsection \ref{3.1.2} for notation), then it is multiplicity free. Consequently, with same notation $ r_{\M}^{\G_n}$, we define the Jacquet functor  $$r_{\M}^{\G_n}:\mathcal{M}_{\Irr}(\G_n)\rightarrow \mathcal{M}(\M).$$

\begin{theorem}\label{1} The Jacquet functor $r_{\M}^{\G_n}:\mathcal{M}_{\Irr}(\G_n)\rightarrow \mathcal{M}(\M)$ is multiplicity free, i.e.,
	$$ \Dim_{\mathbb{C}}\Hom_{\M}(r_{\M}^{\G_n}(\pi), 
	\rho) \leq 1,$$
	for $\pi \in \mathcal{M}_{\Irr}(\G_n)$ and any irreducible representation $\rho\in \mathcal{M}(\M)$.
\end{theorem}
\subsection{Remark} This result is more or less an 
observation that the semi-simplified Jacquet module of any $\pi \in \mathcal{M}_{\Irr}(\G_n)$ is multiplicity free. While broader classes of representations may yield similar results, our focus lies in establishing this property specifically for the class $\mathcal{M}_{\Irr}(\G_n)$. Furthermore, while this method is applicable to other reductive groups, our investigation concentrates solely on the general linear group.
\subsection{Organization} Section 2 gives the 
necessary preliminaries for the general linear group and introduces the notations we use throughout the article. In this section, we recall the geometric lemma. Section  $3$ introduces certain irreducible representations and discusses the irreducibility criteria for such representations. We conclude this section by computing the Jacquet module of some representations and proving Theorem \ref{1}.

\section{Preliminaries.}
This section introduces some basic notions and notations used throughout the article.
\subsubsection{}
Given $\pi\in \mathcal{M} (\G_n)$, we
denote by $\JH^0 (\pi),$ the set of irreducible quotients of the Jordan-Holder series of $\pi.$ Let $\Irr \G_n$ be the subcategory of $\mathcal{M}(\G_n)$, consisting of equivalence classes of irreducible admissible representations of $\G_n$. 
\subsubsection{}
Let $\eta:\G_n\rightarrow \G_n$ be an inner automorphism, then  define the functor $$\eta:\mathcal{M} (\G_n)\rightarrow \mathcal{M} (\G_n).$$  If $(\rho,\V)\in \mathcal{M} (\G_n),$ then the representation  $(\eta\rho,\V)\in  \mathcal{M} (\G_n)$ is defined by 
$$(\eta\rho)(g)v=\rho(\eta^{-1}g)v$$  for $g\in\G_n,v\in\V.$
\subsubsection{}
Let $\alpha=(n_1, \ldots, n_r)$ be a partition of $n$ and $\G_{\alpha}$ be the subgroup $\G_{n_1} \times \cdots \times \G_{n_r}$ of $\G_n$, embedded in $\G_n$ as the subgroup of block-diagonal matrices.~The sets of indices:

$$\mathcal{I}_1=\{1,2,\ldots ,n_1\},$$  $$\hspace{1.5cm}\mathcal{I}_2=\{n_1+1,
\ldots,n_1+n_2\}, $$ $$\vdots$$ $$\hspace{2.5cm}\mathcal{I}_r=\{n_1+\cdots+n_{r-1}+1,\ldots,n\}$$ are called as blocks of $\alpha$.
We say that partition $\beta$ is a subpartition of $\alpha$ if each block of $\beta$ is contained in some block of
$\alpha$. We denote this by $\beta\leqq\alpha$. 
If $\beta\leqq \alpha ,$ then $\G_{\beta} $ is a standard subgroup of $\G_{\alpha}$. Thus, we have the  defined induction and Jacquet functors:
$$ i_{\alpha,\beta}:\mathcal{M} (\G_\beta)\rightarrow \mathcal{M} (\G_\alpha) $$
$$\hspace{0.2cm}r_{{\beta},{\alpha}}:\mathcal{M} (\G_{\alpha})\rightarrow \mathcal{M} (\G_{\beta}).$$
These functors have the following properties:
\begin{theorem}[{\cite[Proposition 1.1]{z80}}] If $\alpha$, $\beta$ 
	and $\gamma$ be the partitions of $n$, then
	\begin{enumerate}
		\item[\upshape(1)] The functor  $r_{\beta,\alpha} $ is left adjoint to $i_{\alpha,\beta}$.
		\item[\upshape(2)] If $\gamma\leqq\beta \leqq \alpha, $ then $i_{\alpha,\beta} \circ  i_{\beta,\gamma}=i_{\alpha,\gamma}.$
	\end{enumerate}
\end{theorem}
\subsubsection{} If $\alpha=(n)$ and $\beta=(l,n-l),$ then we can simply denote $r_{\beta,\alpha} $ by $r_{\M}^{\G_n},$
where $\M:=\GL_l\times\GL_{n-l}$ is the maximal Levi subgroup of $\G_n.$

\subsubsection{}  Consider the functor of the tensor product
$$\otimes:\mathcal{M} (\G_{n_1})\times\cdots\times \mathcal{M} (\G_{n_r})\rightarrow \mathcal{M} (\G_\alpha)$$ defined by $$(\rho_1,\ldots,\rho_r)\mapsto\rho_1\otimes\cdots\otimes\rho_r. $$  The above functor induces a bijection $$\otimes:\Irr \G_{n_1}\times\cdots\times\Irr \G_{n_r}\hspace{0.2cm}\xrightarrow{\sim}\hspace{0.2cm}  \Irr \G_\alpha.$$  Let $\beta\leqq \alpha$ and $(\beta_1,\ldots,\beta_r)$ be the partition of $\beta$, induced by the blocks of $\alpha$. It is clear that $\G_\beta=\G_{\beta_1}\times\cdots\times \G_{\beta_r}$. 

\begin{proposition}[{\cite[Proposition 1.5]{z80}}]
	Let  $\pi_i \in \mathcal{M} (\G_{n_i})$ and   $\rho_i \in \mathcal{M} (\G_{\beta_i})$ with $~i=1,\ldots,r,$ then
	\begin{enumerate}
		\item[\upshape(1)] 
		$r_{\beta,\alpha}(\pi_1\otimes\cdots\otimes\pi_r)=r_{\beta_1,(n_1)}(\pi_1)\otimes\cdots\otimes r_{\beta_r,(n_r)}(\pi_r).$
		\item[\upshape(2)] 
		$i_{\alpha,\beta}(\rho_1\otimes\cdots\otimes\rho_r)=i_{(n_1),\beta_1}(\rho_1)\otimes\cdots\otimes i_{(n_r),\beta_r}(\rho_r). $
	\end{enumerate}
\end{proposition}
\subsubsection{}     Our results are based on the computation of the functor 
$$\mathcal{F}=r_{\gamma,(n)}\circ i_{(n),\beta}(\rho),$$ where $\rho\in 	\mathcal{M} (\G_{\beta})$, which we summarize now, given by 
Zelevinsky \cite{z80}.
Let $\mathcal{I}_1, \ldots, \mathcal{I}_r$ and $\mathcal{J}_1,\ldots,\mathcal{J}_s$ be the blocks of the partitions $\beta$ and $\gamma,$ respectively. Let $ W = W_n$ be the group of all permutations of the set $[1, n]$. We can identify each element $w \in W$ with matrix $w = (\delta_{i,w(j)})\in \G_n$. Therefore, $W$ is a subgroup of $ \G_n$, called the Weyl
group. For any $w\in W$, we denote by the same symbol $w$ the corresponding inner automorphism of $\G_n$, i.e., $w(g)=wgw^{-1}$. Set
\begin{multline*}
	W^{\beta,\gamma}=\{w\in W:w(k)<w(l)\hspace{0.1cm} \text{if} \hspace{0.1cm}  k < l$ $ \hspace{0.1cm} 
	\text{and both $k$ and $l$ belongs to the same} \hspace{0.1cm} \mathcal{I}_i;\\~ w^{-1}(k)<w^{-1}(l) \hspace{0.1cm} \text{if} \hspace{0.1cm} k<l\text{ and both $k$ and $l$ belongs to the same} \hspace{0.1cm} \mathcal{J}_j \}.
\end{multline*}
Clearly all sets $w(\mathcal{I}_i)\cap \mathcal{J}_j$ are blocks for $w\in W^{\beta,\gamma}$. They induce the subpartition $\gamma'=\gamma\cap w(\beta))\leqq \gamma.$     Similarly, the sets $\mathcal{I}_i\cap w^{-1}(\mathcal{J}_j)$ are blocks of the partition $\beta'=\beta\cap w^{-1}(\gamma)\leqq\beta.$  It is clear that $w(\G_{\beta'})=\G_{\gamma'}.$ For a given $w\in W^{\beta,\gamma} $, define the functor $$\mathcal{F}_w: \mathcal{M} (\G_{\beta})\rightarrow \mathcal{M} (\G_{\gamma}),$$ by
$$\mathcal{F}_w=i_{\gamma,\gamma'}\circ w\circ r_{\beta',\beta}.$$

\begin{theorem}[{\cite[Theorem 1.2]{z80}}]\label{10}
	The functor $$\mathcal{F}=r_{\gamma,(n)}\circ i_{(n),\beta}: \mathcal{M} (\G_{\beta})\rightarrow  \mathcal{M} (\G_{\gamma})$$ is glued together from the functors $\mathcal{F}_w$, where $w\in W^{\beta,\gamma}$.
\end{theorem}

\subsubsection{} Now, we explain the representation $\mathcal{F}_w(\rho)$, where $\rho=\rho_1\otimes\cdots\otimes\rho_r \in$ $\mathcal{M} (\G_{\beta})$ and $\rho_i\in$ $\Irr \G_{n_i}$.  Fix the partitions ~$\beta=(n_1,\ldots,n_r) ~\text{and} ~\gamma= 
(m_1,\ldots,m_s$) of $n$. Let $\mathcal{I}_1, \ldots , \mathcal{I}_r$ and $\mathcal{J}_1,\ldots,\mathcal{J}_s$ be the blocks of $\beta$ and $\gamma$, respectively. Therefore, $n_i=|\mathcal{I}_i|$ and $~m_j=|\mathcal{J}_j|.$ To each $w\in W^{\beta,\gamma},$ there corresponds the rectangular $ r\times s $-matrix $B(w)=(b_{ij})$, where $b_{ij}=|\mathcal{I}_i\cap w^{-1}(\mathcal{J}_j)|$. It is clear that the correspondence $w\mapsto B(w)$ is the bijection of $W^{\beta,\gamma} $ with the set $M^{\beta,\gamma}$ of matrices $B=(b_{ij})$ such that:
\begin{itemize}
	\item All $b_{ij}$ are non-negative integers.
	\item $\sum_{j=1}^s b_{ij}=n_i $ for any $i=1,\ldots,r;~ \sum_{i=1}^r b_{ij}=m_j$ for any $j=1,\ldots,s$.
\end{itemize}
We compute composition factors of $\mathcal{F}_w(\rho)$ in terms of the matrix $B(w)=(b_{ij})\in M^{\beta,\gamma} $.
Denote by $\beta_i,$ the partition $(b_{i1},\ldots, b_{is})$ of $n_i$ and by $\gamma_j,$ the partition $(b_{1j},\ldots,b_{rj} )$ of $m_j $. Since representations $r_{\beta_i,(n_i)}(\rho_i)$ have finite length, let $$\JH^0(r_{\beta_i,(n_i)}(\rho_i)= \{\sigma_i^{(1)},\sigma_i^{(2)},\ldots,\sigma_i^{(t_i)}\},$$ where
\begin{equation}\label{gem1}
	\sigma_i^{(k)}=\sigma_{i1}^{(k)}\otimes\cdots\otimes\sigma_{is}^{(k) }; 
\end{equation}
with $ \sigma_{ij}^{(k)}\in \Irr\G_{b_{ij}} ~ \text{and} ~k\in[1,t_i]~\text{for} ~i=1,2,\ldots ,r.$

For each $k_1,\ldots,k_r$; $k_i\in[1,t_i],$ put 
\begin{equation}\label{gem2}
	\sigma_j=\sigma_{1j}^{(k_1)}\otimes\cdots\otimes\sigma_{rj}^{(k_r) } \in \Irr\G_{\gamma_j}
\end{equation} and
\begin{equation}\label{g}
	\sigma(k_1,\ldots,k_r)=i_{(m_1),\gamma_1}(\sigma_1)\otimes i_{(m_2),\gamma_2}(\sigma_2)\otimes\cdots\otimes i_{(m_s),\gamma_s}(\sigma_s) \in \mathcal{M} (\G_\gamma).  
\end{equation}

\begin{proposition}\label{3}
	$\mathcal{F}_w(\rho)$ is glued together from the representations $\sigma(k_1,\ldots,k_r).$
\end{proposition}
According to the above mentioned theorem, when $\beta \nleqq \gamma$, the Jacquet module of a representation induced by an irreducible supercuspidal representation can be expressed as follows.
\begin{corollary}\label{gem3}
	Let $\beta=(n_1,\ldots,n_r) $ be a partition of $n$. Let $\pi=\rho_1\otimes\cdots\otimes\rho_r$ be a representation of $\G_\beta,$ where $\rho_i$'s are the irreducible supercuspidal representations of $\G_{n_i}$. Then $r_{\gamma,(n)}(i_{(n),\beta}(\pi))=0$ if $\beta\nleqq\gamma. $
\end{corollary}
\begin{proof}
	If $\beta\nleqq\gamma, $ then $r_{\beta_i,(n_i)}(\rho_i)=0$ for some $i\in [1,r]$. This implies $\sigma_j=0$~(as in Equation \ref{gem2}). Therefore, by Equation \ref{g}, $\sigma(k_1,\ldots,k_r)=0$. So, by Proposition \ref{3} and Theorem  \ref{10}, $\mathcal{F}(\pi)= r_{\gamma,(n)}(i_{(n),\beta}(\pi))=0.$  
\end{proof}
\section{Representations of the general linear group.}\label{section3}
\subsection{}\label{3.1} 
In this section, our foremost objective is to comprehend the category of representations under discussion thoroughly.

\subsubsection{}\label{3.1.1} Let $\wp$  be the set of equivalence classes of 
irreducible supercuspidal representations of the groups $\G_n,  n\in \mathbb{N}$. Let $\nu$: $\nu (g) = | \Det(g)|$ denotes a character of $\G_n.$ Call a segment in $\wp$, any subset of $\wp$ of the form $$\Delta = [\nu^a\rho,\nu^b\rho] = \{\nu^a \rho, \nu^{a+1} \rho, \nu^{a+2}\rho, \ldots , \nu^b \rho  \},$$ where $a\leq b$ are integers. To each segment $\Delta =[\nu^a\rho,\nu^b\rho]$, we associate the irreducible representation $Z( \Delta)$, which is  defined as the unique irreducible subrepresentation of $\nu^a\rho \times \nu^{a+1}\rho \times \cdots \times \nu^b\rho$.
\begin{definition}
	Let $ \Delta_1=[\nu^{a_1}\rho, \nu^{b_1}\rho]$  and $ \Delta_2=[\nu^{a_2}\rho, \nu^{b_2}\rho]$ be two segments in $\wp$. We say that $\Delta_1$ and $\Delta_2 $ are linked
	if $ \Delta_1\not\subset \Delta_2$, $\Delta_2\not\subset\Delta_1 $ and $\Delta_1\cup\Delta_2 $ is also a segment.
	For the segment $ \Delta=[\nu^{a}\rho, \nu^{b}\rho]$, $b-a+1$ is called length of segment $\Delta.$
\end{definition}

\begin{definition}
	If $ \Delta_1=[\nu^{a_1}\rho, \nu^{b_1}\rho]$  and $ \Delta_2=[\nu^{a_2}\rho, \nu^{b_2}\rho]$ are two segments. Then we say $\Delta_1$ precedes $\Delta_2$  if we can extract a subsequence from the sequence $(\nu^{a_1}\rho,\ldots,\nu^{b_1}\rho,\ldots,\nu^{b_2}\rho)$, which is a segment of
	length strictly greater than length of $\Delta_1$ and $\Delta_2.$
	
\end{definition}
We have a straightforward lemma regarding the non-linkedness of two smaller segments.
\begin{lemma}\label{remark}
	Let $ \Delta_1=[\nu^{a_1}\rho, \nu^{b_1}\rho]$  and $ \Delta_2=[\nu^{a_2}\rho, \nu^{b_2}\rho]$ be two non-linked segments such that $\Delta_1\not\subset\Delta_2$ and $\Delta_2\not\subset\Delta_1$. Then  the segments $ [\nu^{a_1}\rho, \nu^{c_1}\rho]$  and $ [\nu^{a_2}\rho, \nu^{c_2}\rho]$ are also not linked for any $a_i\leq c_i\leq b_i,~i=1,2.$ Also the segments $ [\nu^{c_1}\rho, \nu^{b_1}\rho]$  and $ [\nu^{c_2}\rho, \nu^{b_2}\rho]$ are not linked.
\end{lemma}
Now,  we recall a criterion for the irreducibility of the representation $Z( \Delta_1)\times \cdots\times Z( \Delta_r)$, given by Zelevinsky \cite{z80}.

\begin{theorem}\label{linked}
	Let $\Delta_1,\ldots,\Delta_r$ be the segments in $\wp$. The following are equivalent:
	\begin{enumerate}
		\item[\upshape(1)] The representation $Z( \Delta_1)\times \cdots\times Z( \Delta_r) $ is irreducible.
		\item[\upshape(2)]	 The segments $\Delta_i$ and $\Delta_j$ are not linked for each $i, j= 1, \ldots, r,$
	\end{enumerate}
\end{theorem}

\subsubsection{}
The following theorem analyzes the isomorphism between two induced representations of type $Z( \Delta_1)\times \cdots \times Z( \Delta_r) $, given by Zelevinsky \cite{z80}.

\begin{theorem}\label{5}
	Let $(\Delta_1,\ldots,\Delta_r)$ and $(\Delta_1',\ldots,\Delta_r') $ be the ordered sequences of segments
	in $\wp$. If any of the following conditions holds:
	\begin{enumerate}
		\item $(\Delta_1,\ldots,\Delta_r)$ differs from $(\Delta_1',\ldots,\Delta_r') $ only by a transposition of two neighbors which
		are not linked.
		\item Both  $(\Delta_1,\ldots,\Delta_r)$  and $(\Delta_1',\ldots,\Delta_r') $ satisfy the condition: For each pair of indicies $i,j$ such that $i<j$, $\Delta_i$ does not precedes $\Delta_j $ (same for $\Delta_i'$).
	\end{enumerate}	
	Then $Z( \Delta_1)\times \cdots\times Z( \Delta_r) \simeq Z( \Delta_1')\times \cdots\times Z( \Delta_r') $.
\end{theorem}
\subsubsection{}\label{3.1.2}
%Let $\mathcal{M}(\G_n)$ denote the category of smooth representations of the form $Z( \Delta_1)\times \cdots\times Z( \Delta_r)$. From this point forward, we will focus exclusively on representations falling within the category $\mathcal{M}(\G_n)$. 
Let $\mathcal{M}_{\Irr}(\G_n)$ denote the subcategory of smooth irreducible representations of the form $Z( \Delta_1)\times \cdots\times Z( \Delta_r)$ with $\Delta_i\not\subset\Delta_j$ for $i\neq j$ and $i,j=1,2,\ldots,r$.
% For maximal Levi subgroup $\M:= \G_l \times \G_{n-l}$, let $\mathcal{M}_{\Irr}(\M)$ denote the category of representations of type $\bigoplus_{i,j}\tau_i\otimes\tau_j$, where $\tau_i\in\mathcal{M}_{\Irr}(\G_{l})$, $\tau_j\in\mathcal{M}_{\Irr}(\G_{n-l})$ and $i,j$ ranges over finite natural numbers.

Since we aim to compute the Jacquet module of any $\pi=Z( \Delta_1)\times\cdots\times Z( \Delta_r)\in\mathcal{M}(\G_n),$ we recall the Jacquet module of $Z(\Delta)$, given by Zelevinsky \cite{z80}.  Computation of the Jacquet module of representation $Z( \Delta_1)\times \cdots\times Z( \Delta_r)$ heavily depends on the following proposition. Therefore, we provide the proof also. 

\begin{proposition}\label{single}
	
	Let $\rho$ be an irreducible supercuspidal representation of $\G_m$ and $\Delta = [\nu^a\rho, \nu^b\rho]$ be a segment in $ \wp,$ where $a\leq b$ are integers such that $m(b-a+1)=n$. Then
	$$r_{(l,n-l),(n)}(Z( \Delta) )=
	\begin{cases}
		Z( [\nu^a\rho,\nu^{p-1}\rho])\otimes Z( [\nu^p\rho,\nu^b\rho]) & \text{if } l=mp \\
		0 & \text{if } m\nmid l\end{cases}.$$
\end{proposition}

\begin{proof}
	Let $\pi=\nu^a\rho\times \nu^{a+1} \rho\times\cdots\times\nu^b\rho$ and  $ \beta=(m,m,\ldots,m)$ be a partition of $n$. The Corollary \ref{gem3} implies   $r_{\gamma ,(n)}(\pi)=0$ if $\beta \nleqq \gamma$.~In particular $r_{(l,n-l),(n)}(\pi)=0$ if $l$ is not divisible by $m$. This gives $r_{(l,n-l),(n)}(Z( \Delta) )=0$ for $m\nmid l$.
	Let $l=mp$ and $\gamma =(l,n-l)$ be a partition of $n.$ It easily follows that $r_{\beta,\gamma}(\sigma)\neq 0$ for each  composition factor $\sigma$ of  $r_{\gamma,(n)}(Z( \Delta) )$. Since
	$$r_{\beta,\gamma}\circ r_{\gamma,(n)}(Z(  \Delta))=r_{\beta,(n)}(Z(  \Delta))=\nu^a\rho\otimes \nu^{a+1}\rho\otimes\cdots\otimes\nu^b\rho$$
	is irreducible, $r_{\gamma,(n)}(Z(  \Delta))$ is irreducible. Therefore
	$$r_{\gamma,(n)}(Z(  \Delta))=\sigma_1\otimes\sigma_2;\hspace{1cm}\sigma_1\in \Irr\G_l,~\sigma_2\in \Irr\G_{n-l}.$$
	Moreover,
	$$ r_{\beta,\gamma}(\sigma_1\otimes\sigma_2)= r_{(m,\ldots, m),(l)}(\sigma_1)\otimes r_{(m,\ldots,m),(n-l)}(\sigma_2).\hspace{1.0cm}$$
	Hence,
	$$r_{(m,\ldots, m),(l)}(\sigma_1)=\nu^a\rho\otimes \nu^{a+1}\rho\otimes\cdots\otimes \nu^{p-1}\rho\hspace{1.3cm}$$
	and
	$$r_{(m,\ldots,m),(n-l)}(\sigma_2)=\nu^p\rho \otimes\cdots\otimes\nu^b\rho.\hspace{2.5cm} $$
	Note that $r_{\beta,(n)}(Z(\Delta))=\nu^a\rho\otimes\cdots\otimes\nu^b\rho.$
	Consequently, $$\sigma_1=Z( [\nu^a\rho,\nu^{p-1}\rho]),\sigma_2= Z(  [\nu^p\rho,\nu^b\rho]).$$  The proposition follows. 
\end{proof}
An immediate consequence of the above proposition can be established concerning the multiplicity-free Jacquet functor, providing insight into its behavior under certain conditions. Let $\Par_n$ be the subgroup of $\G_n$, consisting of the matrices having bottom row $\begin{pmatrix}
	0 & 0 & \cdots & 0 & 1
\end{pmatrix}$. Now, we prove the multiplicity-free Jacquet functor for certain types of representations.\\

\begin{theorem}
	If $\pi$ is an irreducible representation of  $\G_n$ such that it is irreducible as $\Par_n$-module, then the Jacquet module of $\pi$ with respect to maximal Levi subgroup is multiplicity-free.
\end{theorem}
\begin{proof}
	
	If $\pi$ is any irreducible representation of  $\G_n$ such that it is irreducible as $\Par_n$-module, then by {\cite[Corollary 7.9]{z80}}, $\pi=Z(\Delta)$ for some segment $\Delta=[\nu^a\rho,\nu^b\rho]$, where $\rho$ is an irreducible supercuspidal representation of $\G_m$ and $a\leq b$ are integers such that $m(b-a+1)=n$. Let $\M=\G_l\times\G_{n-l}$ be a maximal Levi subgroup of $\G_n.$ Then by Proposition \ref{single},
	$$r_{\M}^{\G_n}(Z( \Delta) )=
	\begin{cases}
		Z( [\nu^a\rho,\nu^{p-1}\rho])\otimes Z( [\nu^p\rho,\nu^b\rho]) & \text{if } l=mp \\
		0 & \text{if } m\nmid l\end{cases}.$$
	It is easy to see that all components of $r_{\M}^{\G_n}(Z( \Delta) )$ are distinct. This implies that for any irreducible representations $\rho_1$ and $\rho_2$ of $\G_l$ and $\G_{n-l}$, respectively, $$\Dim_{\mathbb{C}} \Hom_{\M}(r_{\M}^{\G_n}(Z(\Delta)), \rho_1\otimes\rho_2) \leq 1.$$
	Hence, the result holds for $\pi.$
\end{proof}

\subsection{Jacquet module calculation}\label{section}
Let $\Delta_i=[\nu^{a_i}\rho_i,\nu^{b_i}\rho_i]$ be the segments, where $\rho_i$ is an irreducible supercuspidal representation of $\G_{m_i},$ for $i=1,2,\ldots,r$. Then $Z( \Delta_i)$ is an irreducible representation of $\G_{n_i}=\G_{k_im_i}$, where $k_i=b_i-a_i+1$ is length of segment $[\nu^{a_i}\rho_i,\nu^{b_i}\rho_i]$. Therefore,~$Z( \Delta_1)\times\cdots\times Z( \Delta_r) $ is a representation of $\G_n$.  First, we compute the Jacquet module of the representation of the form $Z( \Delta_1)\times\cdots\times Z( \Delta_r)$, i.e.,  $$r_{(l,n-l),(n)}(Z( \Delta_1)\times\cdots\times Z( \Delta_r)).$$
Note that here $\beta=(n_1,n_2,\ldots,n_r)$ and $\gamma=(l,n-l)$ are the the partitions of $n.$
For $ w\in W^{\beta,\gamma},$
we compute the composition
factors of $ \mathcal{F}_w(Z( \Delta_1)\times\cdots\times Z( \Delta_r))$ in terms of the matrix $B(w)=(b_{ij})\in M^{\beta,\gamma} $, where $ (b_{ij})$ is the matrix of order $r\times2$ with entries $(b_{i1})=l_i$ and $(b_{i2})=q_i$ such that $l_i+q_i=n_i,\sum_{i=1}^rl_i=l  $ and $\sum_{i=1}^rq_i=n-l  $.

By Proposition \ref{single}, $\JH^0(r_{(l_i,q_i),(n_i)}(Z( \Delta_i))=0~\text{or}~ \sigma_i^{(w)},$ where
$$\hspace{1.4cm}\sigma_i^{(w)}=Z(  [\nu^{a_i}\rho_i,\nu^{p_i^{(w)}-1}\rho_i]) \otimes Z(  [\nu^{p_i^{(w)}}\rho_i,\nu^{b_i}\rho_i])$$  where
\begin{itemize}
	\item $ a_i-1\leq p_i^{(w)}\leq b_i+1$
	\item $l_i=m_i(p_i^{(w)}-a_i)$
	\item $q_i=m_i(b_i-p_i^{(w)}+1)$
\end{itemize} $~\text{for}~ i=1,2,\ldots,r.$
Let
$$\sigma_1^{(w)}= Z(  [\nu^{a_1}\rho_1,\nu^{p_1^{(w)}-1}\rho_1]) \otimes\cdots\otimes Z(   [\nu^{a_r}\rho_r,\nu^{p_r^{(w)}-1}\rho_r])\hspace{0.3cm}$$\vspace{0.1cm} and
$$\sigma_2^{(w)}= Z( [\nu^{p_1^{(w)}}\rho_1,\nu^{b_1}\rho_1])\otimes\cdots\otimes Z( [\nu^{p_r^{(w)}}\rho_r,\nu^{b_r}\rho_r]).\hspace{0.7cm}$$
Then, by Equation \ref{g}, \begin{align}\label{7}
	\sigma(w)=i_{(l),(l_1,\ldots,l_r)}(\sigma_1^{(w)})\otimes i_{(n-l),(q_1,\ldots,q_r)}(\sigma_2^{(w)}).\hspace{1.3cm}
\end{align}
Therefore, by Proposition  \ref{10},
\begin{align}\label{8}
	r_{(l,n-l),(n)}(Z( \Delta_1)\times\cdots\times Z( \Delta_r))= \sum_{B(w)\in M^{\beta,\gamma}} \sigma(w).
\end{align}	

We computed the Jacquet 
module of any $\pi=Z( \Delta_1)\times\cdots\times Z( \Delta_r)\in\mathcal{M}(\G_n).$ Now, we prove that if the Jacquet functor $r_{\M}^{\G_n}$ is restricted to $\mathcal{M}_{\Irr}(\G_n)$, then it is multiplicity free, i.e., the Jacquet functor $r_{\M}^{\G_n}:\mathcal{M}_{\Irr}(\G_n)\rightarrow \mathcal{M}(\M)$ is multiplicity free.

\subsection{Proof of the Theorem \ref{1}}
\begin{proof}
	Consider $\pi=Z( \Delta_1)\times\cdots\times Z( \Delta_r)\in\mathcal{M}_{\Irr}(\G_n)$ and an irreducible representation $\rho=\tau_1\otimes\tau_2 \in\mathcal{M}(\M)$, where  $\tau_1\in\mathcal{M}_{\Irr}(\G_l) 
	$
	and $\tau_2\in \mathcal{M}_{\Irr}(\G_{n-l}) $.
	Keeping the notations of  \S\ref{section}, by Theorem \ref{linked}, all $\Delta_i $'s are not linked as $\pi$ is an irreducible representation. Hence, by Lemma \ref{remark}, the segments
	$[\nu^{a_i}\rho_i,\nu^{p_i^{(w)}-1}\rho_i]$ and $  [\nu^{p_i^{(w)}}\rho_i,\rho_i^{b_i}]$ are also not linked. As a result,
	$$Z(  [\nu^{a_1}\rho_1,\nu^{p_1^{(w)}-1}\rho_1]) \times\cdots\times Z(   [\nu^{a_r}\rho_r,\nu^{p_r^{(w)}-1}\rho_r])=\sigma_1^{(w)'} ~(\text{say})$$ and
	$$Z( [\nu^{p_1^{(w)}}\rho_1,\nu^{b_1}\rho_1])\times\cdots\times Z( [\nu^{p_r^{(w)}}\rho_r,\nu^{b_r}\rho_r])=\sigma_2^{(w)'}~(\text{say})$$ are irreducible representations of $\G_l$ and $\G_{n-l},$ respectively. The representations $\sigma_i^{(w)'}$ for $i=1,2;$ depends on $B(w)\in M^{\beta,\gamma}$ and all matrices $B(w)$  are distinct. This implies that all the factors in  $r_{(l,n-l),(n)}(Z( \Delta_1)\times\cdots\times Z( \Delta_r))$ are distinct and has no repetitions. So, by the Theorem  \ref{5}, we have $$\Dim_{\mathbb{C}}\Hom_{\G_l\times \G_{n-l}}\Bigl(\sum_{B(w)\in M^{\beta,\gamma}} \sigma(w),\tau_1\otimes\tau_2\Bigr)$$ $$\hspace{1cm}=\sum_{B(w)\in M^{\beta,\gamma}} \Dim_{\mathbb{C}}\Hom_{\G_l\times \G_{n-l}}\Bigl( \sigma(w),\tau_1\otimes\tau_2\Bigr)\leq 1.$$	
	Therefore, the theorem follows.
\end{proof}

\subsection{Acknowledgement:} The first author gratefully 
acknowledges the Ministry of Human Resource Development (MHRD), Govt. of India, for providing the necessary funding and fellowship to pursue this research work. 
\section{Declaration}
\textbf{Declaration for Research Article:}
We, the authors of the research article ``A Note On Jacquet Modules of General Linear Groups", declare that this article represents original work and has not been published previously, nor is it under consideration for publication elsewhere. We take full responsibility for the content of this article and affirm that it does not contain any defamatory or unlawful material.

\bibliography{jacq}
\bibliographystyle{alpha}
\end{document}